\documentclass[twoside, 12pt]{article}
\usepackage{amssymb, amsmath, mathrsfs, amsthm}
\usepackage{graphicx}
\usepackage{color}
\usepackage[top=2cm, bottom=2cm, left=2cm, right=2cm]{geometry}
\usepackage{float, caption, subcaption}

\input{epsf}

\newcommand{\bd}{\begin{description}}
\newcommand{\ed}{\end{description}}
\newcommand{\bi}{\begin{itemize}}
\newcommand{\ei}{\end{itemize}}
\newcommand{\be}{\begin{enumerate}}
\newcommand{\ee}{\end{enumerate}}
\newcommand{\beq}{\begin{equation}}
\newcommand{\eeq}{\end{equation}}
\newcommand{\beqs}{\begin{eqnarray*}}
\newcommand{\eeqs}{\end{eqnarray*}}

\definecolor{DarkGreen}{rgb}{0.2, 0.6, 0.3}

\newtheorem{theorem}{Theorem}[section]

\newtheorem{lemma}{Lemma}[section]

\newtheorem{definition}{Definition}

\newtheorem{claim}{Claim}
\newtheorem{fact}{Fact}
\newtheorem{proposition}{Proposition}[section]
\newtheorem{observation}{Observation}[section]

\newtheorem{problem}{Problem}[section]

\begin{document}
\title{\textbf{Gallai-Ramsey Multiplicity\footnote{Supported by the National Science Foundation of China
        (Nos. 12061059).}}}

\author{
Yaping Mao\footnote{School of Mathematics and Statistics, Qinghai
Normal University, Xining, Qinghai 810008, China. {\tt
maoyaping@ymail.com}}\footnote{Academy of Plateau Science and
Sustainability, Xining, Qinghai 810008, China}}

\date{}
\maketitle

\begin{abstract}

Given two graphs $G$ and $H$, the \emph{general $k$-colored
Gallai-Ramsey number} $\operatorname{gr}_k(G:H)$ is defined to be
the minimum integer $m$ such that every $k$-coloring of the complete
graph on $m$ vertices contains either a rainbow copy of $G$ or a
monochromatic copy of $H$. Interesting problems arise when one asks
how many such rainbow copy of $G$ and monochromatic copy of $H$ must
occur. The \emph{Gallai-Ramsey multiplicity}
$\operatorname{GM}_{k}(G,H)$ is defined as the minimum total number
of rainbow copy of $G$ and monochromatic copy of $H$ in any exact
$k$-coloring of $K_{\operatorname{gr}_{k}(G,H)}$. In this paper, we
give upper and lower bounds
for Gallai-Ramsey multiplicity involving some small rainbow subgraphs.\\[2mm]
{\bf Keywords:} Coloring; Ramsey Number; Ramsey Multiplicity;
Gallai-Ramsey Number; Gallai-Ramsey Multiplicity.\\[2mm]
{\bf AMS subject classification 2020:} 05C15; 05C30; 05C55.
\end{abstract}

\section{Introduction}

All graphs in this paper are undirected, finite and simple. In this
work, we consider only edge-colorings of graphs. Let $G=(V(G),E(G))$
be a graph with vertex set $V(G)$ and edge set $E(G)$, where
$E(G)\subseteq 2^{V(G)}$. Let $P_n$ and $C_n$ denote the path and
the cycle on $n$ vertices. A $k$-edge-coloring is \emph{exact} if
all colors are used at least once. A coloring of a graph is called
\emph{rainbow} if no two edges have the same color. The \emph{color
degree} $d^c(v)$ is the number of different colors that are
presented at $v$, and the \emph{color neighborhood} $CN(v)$ is the
set of different colors that are presented at $v$.

\subsection{Ramsey and local Ramsey numbers}

For given graphs $G_1,\ldots,G_k$ and a graph $F$, we say that $F$
is \emph{Ramsey} for $(G_1,\ldots,G_k)$ and we write $F\rightarrow
(G_1,\ldots,G_k)$ if, no matter how one colors the edges of $F$ with
$k$ colors $1,\ldots, k$, there exists a monochromatic copy of $G_i$
of color $i$ in $F$, for some $1\leq i\leq k$.
\begin{definition}\label{Inequality}
The \emph{$k$-colored Ramsey number} $r(G_1,\ldots,G_k)$ of a
complete graph $F$ is defined as
$r(G_1,\ldots,G_k)=\min\{|V(F)|:F\longrightarrow
(G_1,\ldots,G_k)\}$.
\end{definition}

When $G=G_1=\cdots=G_k$, they are denoted by
$\operatorname{r}_k(G)$.

A \emph{local $k$-coloring} of a graph $H$ is a coloring of the
edges of $H$ in such a way that the edges incident to each vertex of
$H$ are colored with at most $k$ different colors.

Gy{\'a}rf{\'a}s et al. \cite{GLST87} introduced the concept of local
Ramsey number.
\begin{definition}\label{Inequality}
The \emph{local Ramsey number} $\operatorname{r}^k_{loc}(G)$ is
defined as the smallest integer $n$ such that $K_n$ contains a
monochromatic copy of $G$ for every local $k$-coloring of $K_n$.
\end{definition}

Since a $k$-coloring is a special case of a local $k$-coloring, it
is clear that
\begin{equation}          \label{eq1}
\operatorname{r}^k_{loc}(G)\geq \operatorname{r}_k(G).
\end{equation}

\subsection{Gallai-Ramsey numbers}

Colorings of complete graphs that contain no rainbow triangle have
very interesting and somewhat surprising structure. In 1967, Gallai
\cite{Gallai67} first examined this structure under the guise of
transitive orientations. The result was reproven in \cite{GySi04} in
the terminology of graphs and can also be traced to \cite{CaEd97}.
For the following statement, a trivial partition is a partition into
only one part.

\begin{theorem}[\cite{CaEd97,Gallai67,GySi04}]\label{Thm:G-Part}
In any coloring of a complete graph containing no rainbow triangle,
there exists a nontrivial partition of the vertices (called a Gallai
partition), say $H_1,H_2,\ldots,H_t$, satisfying the following two
conditions.

$(1)$ The number of colors on the edges among $H_1,H_2,\ldots,H_t$
are at most two.

$(2)$ For each part pair $H_i,H_j \ (1\leq i\neq j\leq t)$, all the
edges between $H_i$ and $H_j$ receive the same color.
\end{theorem}

\begin{definition}
Given two graphs $G$ and $H$, the \emph{general $k$-colored
Gallai-Ramsey number} $\operatorname{gr}_k(G:H)$ is defined to be
the minimum integer $n$ such that every $k$-coloring of the complete
graph on $n$ vertices contains either a rainbow copy of $G$ or a
monochromatic copy of $H$.
\end{definition}

With the additional restriction of forbidding the rainbow copy of
$G$, it is clear that $\operatorname{gr}_k(G:H)\leq
\operatorname{r}_k(H)$ for any graph $G$.

In an edge-colored graph, define $V^{j}$ as the set of vertices with
at least one incident edge in color $j$ and denote $E^{j}$ to be the
set of edges of color $j$ for a given color $j$.

\begin{theorem}{\upshape \cite{TW07}}\label{thP5}
Let $K_n \ (n\geq 5)$ be edge colored such that it contains no
rainbow $P_5$. Then after renumbering the colors, one of the
following holds:
\begin{itemize}
\item[] $(a)$ at most three colors are used;

\item[] $(b)$ color $1$ is dominant; that is, the vertices can be partitioned into classes $A,V^{2},V^{3},V^{4}\dots$
such that edges within class $V^{j}$ are colored 1 or $j$, edges
meeting $A$ and the edges between classes $V^{j}$, $j\geq 2$ are
colored $1$. It means that the sets $V^{j}$, $j\geq 2$, are
disjoint;

\item[] $(c)$ $K_n-v$ is monochromatic for some vertex $v$;

\item[] $(d)$ there exist three special vertices $v_1, v_2, v_3$ such that $E^{2}=\{v_1v_2\}$, $E^{3}=\{v_1v_3\}$, $E^{4}$ contains $v_2v_3$ plus perhaps some edges incident with $v_1$, and every other edge is in $E^{(1)}$;

\item[] $(e)$ there exist four special vertices $v_1, v_2, v_3, v_4$ such that
$\{v_1v_2\}\subseteq E^{2}\subseteq \{v_1v_2, v_3v_4\}$,
$E^{3}=\{v_1v_3, v_2v_4\}$, $E^{4}=\{v_1v_4, v_2v_3\}$, and every
other edge is in $E^{1}$;

\item[] $(f)$ $n = 5$, $V(K_n)=\{v_1,v_2,v_3,v_4,v_5\}$, $E^{1}=\{v_1v_4, v_1v_5, v_2v_3\}$,
$E^{2}= \{v_2v_4, v_2v_5, v_1v_3\}$, $E^{3}=\{v_3v_4, v_3v_5,
v_1v_2\}$ and $E^{4}=\{v_4v_5\}$.
\end{itemize}
\end{theorem}

For $n\geq 1$, let $G_1(n)$ be a $3$-edge-colored $K_n$ that
satisfies the following conditions: The vertices of $K_n$ are
partitioned into three pairwise disjoint sets $V_1$, $V_2$ and $V_3$
such that for $1\leq i\leq 3$ (with indices modulo $3$), all edges
between $V_i$ and $V_{i+1}$ have color $i$, and all edges connecting
pairs of vertices within $V_{i+1}$ have color $i$ or $i + 1$. Note
that one of $V_1$, $V_2$ and $V_3$ is allowed to be empty, but at
least two of them are non-empty.

\begin{theorem}{\upshape \cite{BMOP, GLST87}}\label{th-Star}
For positive integers $k$ and $n$, if $G$ is a $k$-edge-coloring of
$K_n$ without rainbow $K_{1,3}$, then after renumbering the colors,
one of the following holds.

$(a)$ $k\leq 2$ or $n\leq 3$;

$(b)$ $k=3$ and $G\cong G_1(n)$;

$(c)$ $k\geq 4$ and Item $(b)$ in Theorem \ref{thP5} holds.
\end{theorem}

The following observation is immediate.
\begin{observation}\label{Inequality}
For $k\geq 4$, we have $\operatorname{gr}_{k}(P_{5}:H)\geq
\operatorname{gr}_{k}(K_{1,3}:H)$.
\end{observation}

\subsection{Ramsey and Gallai-Ramsey multiplicities}

If $G$ is a graph without isolated points, and $n$ is a positive
integer, the \emph{multiplicity} $M(G; n)$ is defined as the minimum
number of monochromatic copies of $G$ in any $2$-coloring of the
edges of the complete graph $K_n$. The first published result about
multiplicity appears to be the paper \cite{Goodman59}.

In 1974, Harary and Prins \cite{HaPr74} defined the concept of
Ramsey multiplicity, who were the first to discuss Ramsey
multiplicities systematically.

\begin{definition}
Let $G_1,G_2,...,G_k$ be graphs. The \emph{Ramsey multiplicity}
$\operatorname{R}(G_1,G_2,...,G_k)$ is the smallest possible total
number of $G_1$ in color $1$, $G_2$ in color $2$, $\ldots$, $G_k$ in
color $k$ in any $k$-edge-coloring of $K_r$, where
$r=r(G_1,G_2,...,G_k)$.
\end{definition}

Note that the difference between multiplicity and Ramsey
multiplicity is the restriction that $r=r(G_1,G_2,...,G_k)$ in any
$k$-edge-coloring of $K_r$. For more details on the Ramsey
multiplicity, we refer to the papers \cite{HaPr74, BuRo80, FrRo92,
Jacobson82, RoSu76} and a survey paper \cite{BuRo80} by Burr and
Rosta.

Interesting problems arise when one asks how many such rainbow copy
of $G$ and monochromatic copy of $H$ must occur  in any
$k$-edge-coloring of $K_n$ where $n=\operatorname{gr}_{k}(G,H)$. So
it is natural to generalize the concept to Gallai-Ramsey
multiplicity.
\begin{definition}
The \emph{Gallai-Ramsey multiplicity} $\operatorname{GM}_{k}(G,H)$
is defined as the minimum total number of rainbow copy of $G$ and
monochromatic copy of $H$ in any exact $k$-coloring of
$K_{\operatorname{gr}_{k}(G,H)}$.
\end{definition}

The following observation is immediate.
\begin{observation}\label{Inequality}
For $k\geq 4$, if $\operatorname{gr}_{k}(P_{5}:H)=
\operatorname{gr}_{k}(K_{1,3}:H)$, then
$\operatorname{GM}_{k}(P_{5}:H)\leq
\operatorname{GM}_{k}(K_{1,3}:H)$.
\end{observation}

Li, Broersma, and Wang introduced a parameter $g_k(H,n)$, which is
defined as the minimum number of monochromatic copies of $H$ taken
over all $k$-edge-colorings of $K_n$ without rainbow triangles. For
more details on this topic, we refer to \cite{LBW21}.

\section{Results for Gallai-Ramsey numbers}

Cockayne and Lorimer \cite{CoLo75} got the following result.
\begin{theorem}[\cite{CoLo75}]\label{th-CoLo75}
Let $n_1,n_2,...,n_k$ be positive integers and
$n_1=\max\{n_1,n_2,...,n_k\}$. Then
$$
r(n_1K_{2},n_2K_{2},...,n_kK_{2})=n_1+1+\sum_{i=1}^k(n_i-1).
$$
\end{theorem}

Recently, Zou et al. \cite{ZWLM} derived the following result.
\begin{theorem}[\cite{ZWLM}]\label{coro2-4}
For integers $k\geq 5$ and $k\geq t$, if $H$ is a graph of order
$t$, then
$$\operatorname{gr}_k(P_5:H)=
\begin{cases}
\max{\{\lceil\frac{1+\sqrt{1+8k}}{2}\rceil, 5\}}, & k\geq t+1;\\
t+1, & k=\text{t and H is not a complete graph};\\
(t-1)^2+1, & k=\text{t and H is a complete graph}.
\end{cases}
$$
\end{theorem}

For the Gallai-Ramsey numbers involving a rainbow $5$-path or
$3$-star, and a $n$-matching, we can get their exact values.
\begin{theorem}\label{th-P5}
$(1)$ For $k\geq 3$, we have
\begin{center}
\begin{tabular}{ccccc}
\cline{1-5}
$k$ & $3$ & $[4,\frac{n+3}{2}]$ & $[\frac{n+5}{2},2n]$ & $[2n+1,\infty)$\\[0.1cm]
\cline{1-5}
$\operatorname{gr}_k(P_5:nK_2)$ & $4n-2$ & $3n-1$  & $2n+1$  & $\max\{\lceil\frac{1+\sqrt{1+8k}}{2}\rceil,5\}$, \\[0.1cm]
\cline{1-5}
\end{tabular}
\end{center}

$(2)$ For $k\geq 3$, we have
\begin{center}
\begin{tabular}{ccccc}
\cline{1-5}
$k$ & $3$ &   $[4,\frac{n+3}{2}]$ & $[\frac{n+4}{2},n]$  & $[n+1,\infty)$\\[0.1cm]
\cline{1-5}
$\operatorname{gr}_k(K_{1,3}:nK_2)$ & $4n-2$  & $3n-1$ & $2n$  & $\lceil\frac{1+\sqrt{1+8k}}{2}\rceil$.\\[0.1cm]
\cline{1-5}
\end{tabular}
\end{center}
\end{theorem}

We prove the above theorem by the following lemmas.
\begin{lemma}
For $n\geq 2$,
$\operatorname{gr}_3(P_5:nK_2)=\operatorname{gr}_{3}(K_{1,3}:nK_2)=4n-2$.
\end{lemma}
\begin{proof}
From Theorem \ref{th-CoLo75}, we have
$\operatorname{gr}_{3}(K_{1,3}:nK_{2})\leq
\operatorname{gr}_3(P_5:nK_2))=\operatorname{r}_3(nK_2)=4n-2$. Let
$G$ be a $3$-edge-colored of $K_{4n-3}$ obtained from three cliques
$K_{2n-1}^1,K_{n-1}^2,K_{n-1}^{3}$ with colors $1,2,3$,
respectively. Color the edges between $K_{2n-1}^1$ and $K_{n-1}^2$
with $3$, and color the edges between $K_{n-1}^2$ and $K_{n-1}^3$
with $1$, and color the edges between $K_{n-1}^3$ and $K_{2n-1}^1$
with $2$. Observe that there is neither a rainbow copy of $K_{1,3}$
nor a monochromatic copy of $nK_2$. This means that
$\operatorname{gr}_{3}(K_{1,3}:nK_2)=4n-2$.
\end{proof}

\begin{lemma}
Let $k,n$ be the two integers with $k\geq 4$ and $n\geq 4$. If
$k\leq n\leq 2k-4$, then $\operatorname{gr}_{k}(K_{1,3}:nK_2)=2n$.
\end{lemma}
\begin{proof}
Let $G$ be a $k$-colored $K_{2n-1}$ obtained from $(k-1)$ cliques
$K_{2n-2k+3},K_2^1,...,K_2^{k-2}$ of order $2n-2k+3,2,...,2$ with
colors $2,3,...,k$, respectively, by coloring all the edges among
them by $1$. Since it is a coloring containing no rainbow $K_{1,3}$
and no monochromatic $nK_2$, it follows that
$\operatorname{gr}_{k}(K_{1,3}:nK_2)\geq 2n$.

It suffices to show that $\operatorname{gr}_k(K_{1,3}:nK_2))\leq
2n$. For any $k$-coloring of $K_{2n}$, we suppose that there is no
rainbow $K_{1,3}$. From Theorem \ref{th-Star}, $(b)$ holds. Let
$|V^i|=x_i$ for $i=2,3,...,k$. If $x_2\geq n$, then
$$
2(k-2)\leq \sum_{i=3}^{k}x_i\leq n\leq 2(k-2),
$$
and hence $n=2(k-2)$ and $x_3=x_4=\cdots=x_k=2$. It is clear that
there is a $nK_2$ colored by $1$ from $\sum_{i=3}^{k}V^i$ to $V^2$.
If $x_2\leq n-1$, then $\sum_{i=3}^{k}x_i\geq n+1$. Note that the
edge from any vertex in $V^2$ to any vertex in $\sum_{i=3}^{k}V^i$
is colored by $1$. So we have a $n$-matching, as desired.
\end{proof}

\begin{lemma}
Let $k,n$ be the two integers with $k\geq 5$ and $n\geq 4$. If
$\frac{n+5}{2}\leq k\leq 2n$, then
$\operatorname{gr}_k(P_5:nK_2)=2n+1$.
\end{lemma}
\begin{proof}
Let $G$ be a $k$-colored $K_{2n}$ obtained from a clique $K_{2n-1}$
with color $1$ by adding a new vertex $v$ and the remaining $k-1$
colors appearing on the edges from $v$ to $K_{2n-1}$. Since it is a
coloring containing no rainbow $P_5$ and no monochromatic $nK_2$, it
follows that $\operatorname{gr}_k(P_5:nK_2))\geq 2n+1$.

It suffices to show that $\operatorname{gr}_k(P_5:nK_2))\leq 2n+1$.
For any $k$-coloring of $K_{2n+1}$, we suppose that there is no
rainbow $P_5$. From Theorem \ref{thP5}, one of $(b),(c)$ holds.

Suppose $(b)$ holds. Let $|V^i|=x_i$ for $i=2,3,...,k$. If $x_2\geq
n$, then
$$
2(k-2)\leq \sum_{i=3}^{k}x_i\leq n+1\leq 2(k-2),
$$
and hence $n+1=2(k-2)$ and $x_3=x_4=\cdots=x_k=2$. It is clear that
there is a $nK_2$ colored by $1$ from $\sum_{i=3}^{k}V^i$ to $V^2$.
If $x_2\leq n-1$, then $\sum_{i=3}^{k}x_i\geq n+2$. Note that the
edge from any vertex in $V^2$ to any vertex in $\sum_{i=3}^{k}V^i$
is colored by $1$. So we have $n$ form a $n$-matching, as desired.

If $(c)$ holds, then there exists a $K_{2n+1}-v$ is monochromatic
for some vertex $v$, containing a monochromatic $nK_2$.
\end{proof}

\begin{lemma}
Let $k,n$ be two integers with $n\geq 2$, and $n\geq 2k-3$.

$(1)$ If $k\geq 4$, then $\operatorname{gr}_k(P_5:nK_2)=3n-1$.

$(2)$ If $k\geq 4$, then $\operatorname{gr}_k(K_{1,3}:nK_2)=3n-1$.
\end{lemma}

\begin{proof}
We only give the proof of $(1)$, and $(2)$ can be easily proved. Let
$G$ be a $k$-colored $K_{3n-2}$ obtained from $(k-1)$ cliques
$K_{2n-1},K_{n-2k+5},K_2^1,...,K_2^{k-3}$ of order
$2n-1,n-2k+5,2,...,2$ with colors $2,3,...,k$, respectively, by
coloring all the edges among them by $1$. Observe that this is a
coloring containing no rainbow $P_5$ and no rainbow $K_{1,3}$ and no
monochromatic $nK_2$.

To show $\operatorname{gr}_k(P_5:nK_2)\leq 3n-1$, for any
$k$-coloring of $K_{3n-1}$, we suppose that there is neither a
rainbow $P_5$ nor a monochromatic $nK_2$. From Theorem \ref{thP5},
one of $(b),(c),(f)$ holds(one of $(d),(e)$ also holds if $k=4$).

If $(b)$ holds, then we can assume that $|V^{2}|\geq |V^{3}|\geq
\cdots \geq |V^{k}|$ and $A=V(K_{3n-1})\setminus V^2\cup V^3\cup \ldots\cup V^k$. Let $V_4=\cup_{i=4}^kV^i$. Since
$|A|+|V^{2}|+|V^{3}|+|V_{4}|=3n-1$, it follows that $|V^{2}\cup A|\geq n$. If
$|V^{3}|+|V_{4}|\geq n$, then the edges from $V^{3}\cup V_{4}$ to
$V^{2}\cup A$ form a copy of $nK_2$ with color $1$. We now assume that
$|V^{3}|+|V_{4}|\leq n-1$. Let $|V^{3}|=x$ and $|V_{4}|=y$. Clearly,
there exists a set $X\subseteq V^{2}\cup A$ such that
$x+y=|V^{3}|+|V_{4}|=|X|$. Then the edges from $V^{3}\cup V_{4}$ to
$V^{2}\cup A$ form a copy of $(x+y)K_2$ with color $1$. Let
$U^{2}=V^{2}\cup A-X$. Then $|U^{2}|=3n-1-2x-2y$.

We now assume that there is no $nK_2$ with color $1$ in $K_{3n-1}$.
Suppose that there is exactly a $rK_2$ with color $1$ in $V^{2}\cup A$
such that $r$ is maximum. Let $M=\{u_iv_i\,|\,1\leq i\leq r\}$.
Since the edges from $V^{3}$ to $V_{4}$ can form a $yK_2$ with color
$1$, it follows that $r\leq n-1-y$. Clearly, $V^{2}\cup A-M$ is a clique
$C$ of order $3n-1-x-y-2r$ with color $2$.

\begin{claim}\label{claim1}
For each $u_iv_i \ (1\leq i\leq r)$, the edges from $u_i,v_i$ to $C$
are with colors $1,2$, respectively, or they are both with $2$.
\end{claim}
\begin{proof}
Assume, to the contrary, that there is one edge with color $1$ from
$u_i$ (resp. $v_i$) to $C$, respectively. Then the two edges with
color $1$ together with $M-u_iv_i$ form a $(r+1)$-matching with
color $1$, which contradicts to the fact $r$ is the maximum.
\end{proof}

Let $t$ be the number of vertex pair $u_i,v_i$ such that the edges
from them to $C$ are with color $2$. Let $C'=\{u_i\,|\,1\leq i\leq
t\}\cup \{v_i\,|\,1\leq i\leq t\}$. Then there is a $2t$-matching
with color $2$ from $C'$ to $C$. Without loss of generality, let the
vertex pairs are $u_i,v_i$, where $1\leq i\leq t$. From Claim
\ref{claim1}, the edges from $u_i,v_i$ to $C$ are with color $1,2$,
respectively, where $t+1\leq i\leq r$.

\begin{claim}\label{claim2}
For each $u_iv_i,u_jv_j \ (t+1\leq i\neq j\leq r)$, the edge
$\chi(v_iv_j)=2$.
\end{claim}
\begin{proof}
Assume, to the contrary, that $\chi(v_iv_j)=1$. Then
$M-u_iv_i-u_jv_j+v_iv_j+u_iw_1+u_jw_2$ is a $(r+1)$-matching with
color $1$, where $w_1,w_2\in V(C)$, which contradicts to the fact
$r$ is the maximum.
\end{proof}

From Claim \ref{claim2}, the edges in $\{v_iv_j\,|\,t+1\leq i\neq
j\leq r\}$ form a clique of order $r-t$ with color $2$, say $C^{b}$.
Then there is a $(r-t)$-matching with color $2$ from $C^b$ to $C$.
Let $C''=\{u_i\,|\,t+1\leq i\leq r\}$. Then there is a
$(r-t)$-matching with color $1$ from $C''$ to $C$. Since there is no
$nK_2$ with color $1$ and $|V^2|-|C''|\geq 3n-1-x-y-(r-t)\geq
2n-r+t\geq n$, it follows that $|C''|+|V^3|+|V_4|=r-t+x+y\leq n-1$.
Since $|C^{b}|+|C'|=r+t$ and $C-|C^{b}|-|C'|=3n-1-x-y-3r-t$, it
follows that there is a matching of
$$
r+t+\frac{3n-1-x-y-3r-t}{2}=\frac{3n-1-x-y-r+t}{2}\geq n
$$
edges with color $2$, as desired.
If $(f)$ holds, then $k=4$ and $3n-1=5$ and there is a red
$2$-matching.

For $k=4$, if $(c)$ holds, then there exists a $K_{3n-1}-v$ is monochromatic
for some vertex $v$, containing a monochromatic $nK_2$.
If $(d)$ holds, then since $2+\frac{3n-6}{2}\geq n$, it follows that there exists a red $n$- matching. If $(e)$ holds, then since $4+\frac{3n-9}{2}\geq n$, it follows that there exists a red $n$-
matching.
\end{proof}

\begin{lemma}\label{th2-1}
For two integers $k,n$ with $n\geq 3$ and $k\geq n+1$, we have
$$
\operatorname{gr}_k(K_{1,3}:nK_2)=\left\lceil\frac{1+\sqrt{1+8k}}{2}\right\rceil.
$$
\end{lemma}
\begin{proof}
Let $N_k$ be an integer with
$N_k=\lceil\frac{1+\sqrt{1+8k}}{2}\rceil$. For the lower bound, if
there is a $k$-edge-coloring $\chi$ of a complete graph $K_{N_k-1}$,
then $k\leq \tbinom{N_k-1}{2}$, contradicting with
$N_k=\lceil\frac{1+\sqrt{1+8k}}{2}\rceil$. So, we have
$\operatorname{gr}_k(K_{1,3}:nK_2)\geq N_k$. It suffices to show
that $\operatorname{gr}_k(K_{1,3}:H)\leq N_k$. Let $\chi$ be any
$k$-edge-coloring of $K_n \ (n\geq N_k)$ containing no rainbow copy
of $K_{1,3}$. From Theorem \ref{th-Star}, $(b)$ is true. Let
$V^{2},V^{3},\ldots,V^{k}$ be a partition of $V(K_n)$ such that
there are only edges of color $1$ or $i$ within $V^{(i)}$ for $2\leq
i\leq k$. Then $|V^{k}|\geq 2$ for $2\leq i\leq k$, and hence
$\sum_{i=2}^{k}|V^{k}|\geq 2(k-1)\geq 2n$. Therefore there is a
monochromatic copy of $nK_2$ colored by $1$.
\end{proof}

The Gallai-Ramsey number of a triangle versus a $n$-matching was
given in \cite{FGJM10}.
\begin{theorem}[\cite{FGJM10}]\label{th4-2}
$\operatorname{gr}_{k}(K_{3}:nK_{2})=(k+1)(n-1)+2$.
\end{theorem}

\section{Results for Ramsey multiplicity}

Assume that the edges of $K_n$ are locally $k$-colored with colors
$1,2,...,m$. We can define a partition $\mathcal{P}(K_n)$ on the
vertices of $K_n$ in a natural way as follows. Let
$A_{i_1i_2...i_k}$ denote the set of vertices in $K_n$ incident to
edges of colors $i_1,i_2,...,i_k\in \{1,2,...,m\}$. The vertices
incident to edges with colors in $X\subseteq \{i_1,i_2,...,i_k\}$
and $|X|\leq k-1$ can be distributed arbitrarily in the sets
$A_{i_1i_2,...,i_k}$. Every partition class $A_{i_1i_2,...,i_k}$
induces a $k$-colored complete graph in $K_n$. The following idea is
from \cite{GLST87}.

\begin{theorem}\label{th-partition}
Let $K_n$ be locally $k$-colored with colors $1, 2,..., m$. Then
either $m=k+1$ and
$$
\mathcal{P}(K_n)=\{A_{12...k},A_{12...(k-1)(k+1)},...,A_{2...(k-1)k(k+1)}\}
$$
or $m\geq k+2$ and there exist $(m-1+k)$ colors such that
$$
\mathcal{P}(K_n)=\{A_{12...(k-1)k},A_{12...(k-1)(k+1)},...,A_{12...(k-1)m}\}.
$$
\end{theorem}
\begin{proof}
Let $G^*$ be the hypergraph with vertex set $\{1,2,...,m\}$ and
hyperedge set
$$
e_{i}=\{A_{i_1i_2...i_k}\,|\,i_1,i_2,...,i_k\in \{1,2,...,m\}\}.
$$
If $m=k+1$, then
$$
\mathcal{P}(K_n)=\{A_{12...k},A_{12...(k-1)(k+1)},...,A_{2...(k-1)k(k+1)}\}.
$$
Suppose that $m\geq k+2$. Then we have the following fact.
\begin{claim}\label{Claim3}
For any $A_{i_1i_2...i_k}$, if $i_j\not\in \{1,2,...,k-1\}$ for
$1\leq j\leq k$, then $A_{i_1i_2...i_k}$ does not exist.
\end{claim}
\begin{proof}
Assume, to the contrary, that $A_{i_1i_2,...,i_k}$ exists. Without
loss of generality, let $i_j\neq 1$ for $1\leq j\leq k$. Then
$i_1i_2...i_k\in \{2,3,...,m\}$ and $A_{12...(k-1)a}\neq
A_{i_1i_2,...,i_k}$. Since $m\geq k+2$, it follows that there exists
a color $a$ such that $a\not\in \{i_1i_2...i_k\}$. If $k\leq a\leq
k+1$, then for $A_{12...(k-1)a}$ and $A_{i_1i_2...i_k}$, there
exists two vertices $x,y$ such that the colors in
$\{1,2,...,k-1,a\}$ appearing on $x$ and the colors in
$\{i_1,i_2,...,i_k\}$ appearing on $y$. Whenever the color of the
edge $xy$, the color degree of $x$ is at least $k+1$ or the color
degree of $x$ is at least $k+1$, a contradiction.
\end{proof}
From Claim \ref{Claim3}, there exists $(m-k+1)$ colors such that
$$
\mathcal{P}(K_n)=\{A_{12...(k-1)k},A_{12...(k-1)(k+1)},...,A_{12...(k-1)m}\},
$$
as desired.
\end{proof}

Gy{\'a}rf{\'a}s et al. \cite{GLST87} obtained the following result.
\begin{theorem}[\cite{GLST87}]\label{th-GLST87}
For $n\geq 2$, $\operatorname{r}^2_{loc}(nK_2)=3n-1$.
\end{theorem}

We can derive the local Ramsey number of $n$-matchings with $k$
colors.
\begin{theorem}\label{th-GLST87}
For $n\geq 2$, $\operatorname{r}^k_{loc}(nK_2)=n+1+k(n-1)$.
\end{theorem}
\begin{proof}
From Theorem \ref{th-CoLo75} and (\ref{eq1}), we have
$\operatorname{r}^k_{loc}(nK_2)\geq \operatorname{r}_k(nK_2)=
n+1+k(n-1)$. It suffices to show that
$\operatorname{r}^k_{loc}(nK_2)\leq n+1+k(n-1)$. Let
$t=n+1+k(n-1)=(k+1)(n-1)+2$.

For any $m$-colored $K_{t}$, we first assume that $m=k+1$. From
Theorem \ref{th-partition}, we have
$$
\mathcal{P}(K_n)=\{A_{12...k},A_{12...(k-1)(k+1)},...,A_{2...(k-1)k(k+1)}\}.
$$
When this occurs we use induction on $n$. Clearly, there exist
$(k+1)$ vertices, say $x_{k+1}^1\in A_{12...k},x_{k}^1\in
A_{12...(k-1)(k+1)},...,x_{1}^1\in A_{2...(k-1)k(k+1)}$, such that
their induced subgraph is a clique, say $K_{k+1}^1$, with colors
$1,2,...,k+1$. Let $F_{n-1}=K_t-K_{k+1}^1$. It suffices to show that
there is a monochromatic copy of $(n-1)K_2$ in $F_{n-1}$. Clearly,
there exist $(k+1)$ vertices, say $x_{k+1}^2\in
A_{12...k},x_{k}^2\in A_{12...(k-1)(k+1)},...,x_{1}^2\in
A_{2...(k-1)k(k+1)}$, such that their induced subgraph is a clique,
say $K_{k+1}^2$, with colors $1,2,...,k+1$. Let
$F_{n-2}=F_{n-1}-K_{k+1}^2$. It suffices to show that there is a
monochromatic copy of $(n-2)K_2$ in $F_{n-2}$. Continue this
process, there exist $(k+1)$ vertices, say $x_{k+1}^{n-1}\in
A_{12...k},x_{k}^{n-1}\in A_{12...(k-1)(k+1)},...,x_{1}^{n-1}\in
A_{2...(k-1)k(k+1)}$, such that their induced subgraph is a clique,
say $K_{k+1}^{n-1}$, with colors $1,2,...,k+1$. Let
$F_{1}=F_{2}-K_{k+1}^{n-1}$. Since $|F_{1}|=2$, it follows that
there is a monochromatic copy of $nK_2$, as desired.

Next, we assume that $m\geq k+2$. From Theorem \ref{th-partition},
there exist $(m-1+k)$ colors such that
$$
\mathcal{P}(K_n)=\{A_{12...(k-1)k},A_{12...(k-1)(k+1)},...,A_{12...(k-1)m}\}.
$$

\begin{fact}\label{fact1}
For any pair $x\in A_{12...(k-1)i},y\in A_{12...(k-1)j}$ with $k\leq
i\neq j\leq m$, the edge $xy$ is colored by the colors in
$\{1,2,...,k\}$.
\end{fact}

From Fact \ref{fact1}, $K_t$ is $(k-1)$-colored. From Theorem
\ref{th-CoLo75}, we have
$$
r_{k-1}(nK_{2})=k(n-1)+2<t,
$$
and hence there is a monochromatic copy of $nK_2$.
\end{proof}

A \emph{broom} $B_{m,\ell}$ is a path of length $m$ with a star with
$\ell$ leaves joined at one of the path's leaf vertices and the root
vertex of the star. The following result comes from \cite{GySi04}.

\begin{theorem}[\cite{GySi04}]\label{th-broom}
In every rainbow triangle free coloring of a complete graph, there
is a monochromatic spanning broom.
\end{theorem}

We can give the following Ramsey multiplicity of matchings.
\begin{theorem}\label{th-D-Stripes}
Let $n_1,n_2,...,n_k$ be positive integers and
$n_1=\max\{n_1,n_2,...,n_k\}$. Then
$$
\operatorname{M}_k(n_1K_{2},n_2K_{2},...,n_kK_{2})\leq
\frac{(2n_1)!}{2^{n_1}n_1!}.
$$
\end{theorem}
\begin{proof}
Let $m=n_1+1+\sum_{i=1}^k(n_i-1)$. For the upper bound, let $G$ be a
$k$-edge-coloring of $K_{m}$ obtained from $k$ cliques
$K_{2n_1},K_{n_2-1},K_{n_3-1},...,K_{n_k-1}$ of order
$2n_1,n_2-1,n_3-1,...,n_k-1$ with colors $1,2,....,k$. The
\emph{reduced graph} $H$ of $K_{m}$ is constructed by contracting
each clique $K_{2n_1},K_{n_2-1},K_{n_3-1},...,K_{n_k-1}$ by a single
vertex $u_{i} \ (1\leq i\leq k)$. Note that $H$ is a complete graph
of order $n+1$ and $V(H)=\{u_{i}\,|\,1\leq i\leq k\}$. Color the
edges in $\{u_iu_j\,|\,1\leq j\leq i-1\}$ incident $u_i \ (2\leq
i\leq k)$ with color $i$. Then color all the edges from $K_{n_i-1}$
to $K_{n_j-1}$ with color $i$ in $K_{m}$, where $1\leq i\leq k$ and
$1\leq j\leq i-1$, corresponding to the edge $u_iu_j$ in $H$. For
each $i \ (2\leq i\leq k)$, there is no $n_iK_2$ colored by $i$ in
$K_{m}$, and so the total number of $n_i$-matching in $G$, where
$1\leq i\leq k$, equals the number of $n_1$-matching in $K_{2n_1}$
with color $1$. This number is seen to be $(2n_1)!/2^{n_1}n_1!$.
\end{proof}

\subsection{Results involving a rainbow $4$-path or $3$-star}

We are now in a position to give the results on the Ramsey
multiplicity.
\begin{theorem}
For $n\geq 2$,
$$
n!\leq \operatorname{GM}_3(P_5:nK_2)=\operatorname{M}_3(nK_2)\leq
\frac{(2n)!}{2^{n}n!}.
$$
\end{theorem}
\begin{proof}
The upper bound follows from Theorem \ref{th-D-Stripes}. For the
lower bound, we use induction on $n$. For any $3$-edge-coloring of
$K_{4n-2}$, we let red, blue and green are the three colors. We
first suppose that there is a vertex set $V$ with $|V|=n$, say
$V=\{v_1,v_2,...,v_{n}\}$, such that the red degree of each $v_i$ is
at least $2n-1$, where $1\leq i\leq n$. For $v_1$, there are $n$ red
edges outside $V$ incident to $v_1$, and we choose one, say
$u_1v_1$. For $v_2$, there are $(n-1)$ red edges outside $V$
incident to $v_2$, and we choose one, say $u_2v_2$, such that
$v_2\neq v_1$. Continue this process, there are $n!$ different red
copies of $nK_2$.

We can assume that there is no such vertex set $V$ with $|V|=n$ such
that the red degree of each vertex is at least $2n-1$. This means
that the number of vertices with red degree at least $2n-1$ is at
most $n-1$.
\begin{fact}\label{fact3}
There exists a vertex set $X_1$ such that $|X_1|\geq 3n-1$ and the
red degree of each vertex in $X_1$ is at most $2n-2$.
\end{fact}

Similarly, we have the following fact.
\begin{fact}\label{fact4}
There exists a vertex set $X_i$ such that $|X_i|\geq 3n-1$ and the
blue or green degree of each vertex in $X_i$ is at most $2n-2$.
\end{fact}

From Facts \ref{fact3} and \ref{fact4}, we have $X_1\cap X_2\cap
X_3\neq \emptyset$. Then there exists a vertex $v$ such that the
red, blue, and green degree of $v$ is at most $2n-4$. Let
$x_1,x_2,x_3$ be the red, blue, and green degree of $v$. Since
$x_i\leq 2n-2$ for each $i \ (1\leq i\leq 3)$, it follows that
$x_1+x_2\geq (4n-2)-(2n-2)=2n$ and $x_1+x_3\geq 2n$ and $x_2+x_3\geq
2n$, and hence $x_1+x_2+x_3\geq 3n$, and so $x_i\geq n$ for
$i=1,2,3$. Let $v_1,...,v_n$ be adjacent to $v$ in red, and
$u_1,...,u_n$ be adjacent to $v$ in blue, and $w_1,...,w_n$ be
adjacent to $v$ in green. Choose $Y=\{v,v_1,u_1,w_1\}$. By the
induction hypothesis, $K_{4n-2}-Y$ contains $(n-1)!$ monochromatic
copies of $(n-1)K_2$. If it is red, then there are $n!$
monochromatic copies of $nK_2$ from $K_{4n-2}-Y$ and the edges in
$\{vv_i\,|\,1\leq i\leq n\}$. The same is true if the copies are
blue or green.
\end{proof}

\begin{theorem}
For $n\geq 2$,
$$
\min_{0\leq r\leq 4n-2}\left\{{\lfloor (4n-1-r)/3\rfloor \choose
n}+r(4n-5)\right\}\leq \operatorname{GM}_3(K_{1,3}:nK_2)\leq
(2n-1)!!.
$$
\end{theorem}
\begin{proof}
Note that $\operatorname{gr}_{3}(K_{1,3}:nK_2)=4n-2$. To show the
upper bound, let $G$ be a $3$-edge-colored of $K_{4n-2}$ obtained
from three cliques $K_{2n}^1,K_{n-1}^2,K_{n-1}^{3}$ with colors
$1,2,3$, respectively. Color the edges between $K_{2n}^1$ and
$K_{n-1}^2$ with $3$, and color the edges between $K_{n-1}^2$ and
$K_{n-1}^3$ with $1$, and color the edges between $K_{n-1}^3$ and
$K_{2n}^1$ with $2$.

To show the lower bound, let $\chi$ be the $3$-edge-coloring of
$K_{4n-2}$ with vertex set $V(K_{4n-2})=\{v_i\,|\,1\leq i\leq
4n-2\}$.
\begin{fact}
If the color degree of $v_i$ is $3$, then there are at least $4n-5$
rainbow $3$-stars incident to $v_i$ in $K_{4n-2}$.
\end{fact}
\begin{proof}
Without loss of generality, we assume that $\chi(v_1v_2)=1$ and
$\chi(v_1v_3)=2$ and $\chi(v_1v_4)=3$. For each $j \ (5\leq j\leq
4n-2)$, if $\chi(v_1v_j)=1,2,3$, then the edges in
$\{v_1v_j,v_1v_3,v_1v_4\}$, $\{v_1v_j,v_1v_2,v_1v_4\}$,
$\{v_1v_j,v_1v_2,v_1v_3\}$ form a rainbow star. Then there are at
least $4n-5$ rainbow $3$-stars incident to $v_1$.
\end{proof}

Suppose that there are $r$ vertices, say $v_1,v_2,...,v_r$, with
color degree $3$. Then for any vertex $v_i \ (r+1\leq i\leq 4n-2)$,
the color degree of $v_i$ is at most $2$. Let
$F=K_{4n-2}-\{v_i\,|\,1\leq i\leq r\}$. From Theorem
\ref{th-GLST87}, there is a monochromatic copy of $((4n-1-r)/3)K_2$.
Then the total number of monochromatic copy of $nK_2$ and rainbow
copy of $K_{1,3}$ is at least
$$
\min_{0\leq r\leq 4n-2}\left\{{\lfloor (4n-1-r)/3\rfloor \choose
n}+r(4n-5)\right\}.
$$
\end{proof}

The following lemma will be used later.
\begin{lemma}\label{lem-count}
Let $G$ be a $k$-colored $K_{2k-2}$ with a rainbow $(k-1)$-matching
$M^*$ such that all the edges in $G\setminus M^*$ are with color
$1$. If $i\leq k$, then the number of $i$-matching with color $1$ is
$$
\sum_{x=0}^{\lfloor i/2\rfloor}{k-1\choose x}{k-1-x\choose
i-x}2^{x}4^{i-2x}(2x-1)!!(2i-4x-1)!!.
$$
\end{lemma}
\begin{proof}
Let $M^*=\{u_iv_i\,|\,1\leq i\leq k-1\}$ be the rainbow
$(k-1)$-matching. Let $F$ be the reduced graph constructed by
selecting a single vertex $u_i$ from the edge $u_iv_i$, that is, the
subgraph $F$ induced by the vertices in $\{u_i\,|\,1\leq i\leq
k-1\}$ is a clique of order $k-1$. Let $M$ be any $i$-matching with
color $1$.

\begin{fact}\label{fact6}
For $u_iu_j\in E(F)$, if $u_iv_i,u_jv_j\in M$, then
$M'=M-u_iv_i-u_jv_j+u_iv_j+u_jv_i$ is a new matching corresponding
to $u_iu_j\in E(F)$.
\end{fact}

\begin{fact}\label{fact7}
Let $X=\{u_iv_i,u_jv_j,u_iv_j,u_jv_i\}$. For $u_iu_j\in E(F)$, if
$e'\in M$, then $M'=M-e+e'$ is a new matching corresponding to
$u_iu_j\in E(F)$, where $e'\in M$.
\end{fact}

For Fact \ref{fact6}, we assume that there are $x$ such $u_iv_i$s in
$F$. For Fact \ref{fact7}, we assume that there are $y$ such
$u_iv_i$s in $F$. Clearly, $2x+y=i$. Then the number of $i$-matching
with color $1$ is
$$
\sum_{x=0}^{\lfloor i/2\rfloor}{k-1\choose x}{k-1-x\choose
y}2^{x}4^{y}(2x-1)!!(2y-1)!!,
$$
as desired.
\end{proof}

\begin{proposition}
Let $k,n$ be two integers with $n\geq 2$ and $k\geq 5$. If $n\leq
2k-4$, then
$$
\operatorname{GM}_k(P_5:nK_2)\leq X\sum_{i=1}^n{2k-4\choose
n-i}{n\choose n-i}(n-i)!,
$$
where
$$
X=\sum_{x=0}^{\lfloor i/2\rfloor}{k-1\choose x}{k-1-x\choose
i-2x}2^{x}4^{i-2x}(2x-1)!!(2i-4x-1)!!.
$$
\end{proposition}
\begin{proof}
For $k\geq 5$, if $n\leq 2k-4$, then
$\operatorname{gr}_k(P_5:nK_2)=2n$. Let $G$ be a $k$-colored
$K_{2n}$ obtained from the $n-1$ cliques $K_{n},K_2^3,...,K_2^{k}$
be of order $n,2,...,2$ with colors $2,3,...,k$, respectively, by
coloring all the edges among them by $1$. Observe that there is no
$nK_2$ with color $i \ (2\leq i\leq k)$ and there is no rainbow
$P_5$. To get a $nK_2$ with color $1$, we first choose a $(n-i)K_2$
from $V^2$ to $\bigcup_{i=3}^kV^i$ in $G$, and the number of ways is
$\sum_{i=1}^n{2k-4\choose n-i}{n\choose n-i}(n-i)!$. Next, we choose
$iK_2$ in the edges among $\bigcup_{i=3}^kV^i$. From Lemma
\ref{lem-count}, the number of ways is
$$
\sum_{x=0}^{\lfloor i/2\rfloor}{k-1\choose x}{k-1-x\choose
i-2x}2^{x}4^{i-2x}(2x-1)!!(2i-4x-1)!!.
$$
\end{proof}

\begin{theorem}
Let $k,n$ be two integers with $k\geq 5$, $n\geq 11$, and $n\geq
2k-4$. Then
$$
\operatorname{GM}_k(P_5:nK_2)\leq
\sum_{i=1}^{k-3}\sum_{j=1}^{n-2k+6}LM{2n-1\choose
n-i-j}{n-2k+6-2i-2j\choose n-i-j}(n-j-i)!.
$$
where
$$
L=\sum_{x=0}^{\lfloor i/2\rfloor}{k-3\choose x}{k-3-x\choose
i-2x}2^{x}4^{i-2x}(x-1)!!(i-2x-1)!!,
$$
and
$$
M={n-2k+6\choose j}{2k-6-2i\choose j}j!.
$$
\end{theorem}
\begin{proof}
Note that if $n\geq 2k-4$, then
$\operatorname{gr}_k(P_5:nK_2)=3n-1$. Let $G$ be a $k$-colored
$K_{3n-1}$ obtained from $(k-1)$ cliques
$K_{2n-1},K_{n-2k+6},K_2^4,...,K_2^{k}$ of order
$2n-1,n-2k+6,2,...,2$ with colors $2,3,...,k$, respectively, by
coloring all the edges among them by $1$. Let $V^2=K_{2n-1}$, and
$V^3=K_{n-2k+6}$, and $V^i=K_{2}^i$ for $4\leq i\leq k$. Observe
that there is no monochromatic copy of $nK_2$ with color $i \ (2\leq
i\leq k)$. Then we will compute the number of $nK_2$ with color $1$.
We choose a $nK_2$ with color $1$ in $G$ by the following way.

\begin{itemize}
\item A $i$-matching is chosen in the subgraph induced by the vertices in
$\bigcup_{i=4}^kV^i$.

\item A $j$-matching is chosen in the edges from $V^3$ to
$\bigcup_{i=4}^kV^i$.
\end{itemize}

Then we will choose a $(n-i-j)$-matching from the edges between
$V^2$ and $\bigcup_{i=4}^kV^i$. Furthermore, the number of ways to
choose the $jK_2$ is
$$
{n-2k+6\choose j}{2k-6-2i\choose j}j!=M.
$$
From Lemma \ref{lem-count}, the number of ways to choose the $iK_2$
is
$$
\sum_{x=0}^{\lfloor i/2\rfloor}{k-3\choose x}{k-3-x\choose
i-2x}2^{x}4^{i-2x}(x-1)!!(i-2x-1)!!=L,
$$
and hence the number of ways to choose the $nK_2$ is
$$
\sum_{i=1}^{k-3}\sum_{j=1}^{n-2k+6}LM{2n-1\choose
n-i-j}{n-2k+6-2i-2j\choose n-i-j}(n-j-i)!,
$$
as desired.
\end{proof}

\subsection{Results involving a rainbow triangle}

Let $\tau(k,K_n,G)$ denote the number of copies of $G$ in a
$k$-colored $K_n$ containing no rainbow triangle.
\begin{lemma}\label{lemB2K2}
For $m\geq 2$, the number of $2$-matchings in $B_{m,\ell}$ is at
least $(m-1)(m+2\ell-2)$.
\end{lemma}
\begin{proof}
Since $|V(B_{m,\ell})|=m+\ell+1$, it follows that the size of any
spanning tree is $m+\ell$.
\begin{fact}
By deleting the two vertices of an edge $e$, any remaining edge
together with $e$ is a $2$-matching.
\end{fact}

Suppose that there is a broom $B_{m,\ell}$ with center $v_{m+1}$ and leaf vertices $u_i \ (1\leq i\leq \ell)$. For each $u_iv_{m+1} \ (1\leq i\leq \ell)$, by deleting the two
vertices $u_i,v_{m+1}$, the resulting graph is a path of length
$m-1$. The contribution of $2$-matchings is $(m-1)\ell$. For each
$v_1v_{2}$, by deleting the two vertices $v_1,v_{2}$, the resulting
graph is a broom $B_{m-2,\ell}$. The contribution of $2$-matchings
is $m+\ell-2$. For edge $v_{m}v_{m+1}$, by deleting the two vertices
$v_{m},v_{m+1}$, the resulting graph is a path of length $m-2$. The
contribution of $2$-matchings is $m-2$. For each $v_iv_{i+1} \
(2\leq i\leq m-1)$, by deleting the two vertices $v_{i},v_{i+1}$,
the resulting graph has $m+\ell-3$ edges. The contribution of
$2$-matchings is $(m-2)(m+\ell-3)$. Then the total number of
$2$-matchings is
$$
(m-1)\ell+(m+\ell-2)+(m-2)+(m-2)(m+\ell-3)=(m-1)(m+2\ell-2),
$$
as desired.
\end{proof}

\begin{lemma}\label{lemB2K2-2}
For $k\geq 3$,
$$
\tau(k,K_{n},2K_{2})\geq
\min\left\{2(n-3),\tau(1,K_{n-k+1},2K_{2}),
\frac{(2n-k-4)(k-2)}{2}+\tau(1,K_{n-k},2K_{2})\right\}.
$$
\end{lemma}

\begin{proof}
Since there is no rainbow triangle in a $k$-colored $K_{n}$, it
follows from Theorem \ref{th-broom} that there is a spanning broom
$B_{m,\ell}$ with center $v_{m+1}$, where $m+\ell=n-1$. We assume this broom is colored by $1$.
Suppose that $m=1$. If there is no edges with color $1$ in $K_n-v_{m+1}$, then
$\tau(k,K_{n},2K_{2})\geq \tau(k-1,K_{n-1},2K_{2})$, and hence
$\tau(k,K_{n},2K_{2})\geq \tau(1,K_{n-k+1},2K_{2})$. If there is
exactly one edge with color $1$ in $K_n-v_{m+1}$, then
$\tau(k,K_{n},2K_{2})\geq (n-3)+\tau(k-1,K_{n-2},2K_{2})$, and hence
$$
\tau(k,K_{n},2K_{2})\geq
\frac{(2n-k-4)(k-2)}{2}+\tau(k-1,K_{n-2},2K_{2}).
$$ If there is at
least two edges with color $1$ in $K_n-v$, then
$\tau(k,K_{n},2K_{2})\geq 2(n-3)$.
\end{proof}

\begin{observation}\label{obs5-1}
Suppose that $G$ is the $i$-colored $K_4$, where $2\leq i\leq 4$,
containing a monochromatic $2$-matching.

$(1)$ If $i=2$, then there is a monochromatic $C_4$ or $P_4$.

$(2)$ If $i=3$, then there are at least two rainbow triangles.

$(3)$ If $i=4$, then there are at least three rainbow triangles.
\end{observation}

\begin{theorem}
For $k\geq 6$, we have $\operatorname{GM}_{k}(K_{3}:2K_{2})=3$.
\end{theorem}
\begin{proof}
From Theorem \ref{th4-2}, we have
$\operatorname{gr}_{k}(K_{3}:nK_{2})=k+3$. To show the upper bound,
we start at a $1$-colored $K_{4}$, and then repeatedly add vertices
$v_i$ with all edges to $v_i$ having color $i+1$ for $i=1,...,k-1$.
This coloring certainly contains no rainbow triangle but three
monochromatic $2K_{2}$s and so
$\operatorname{GM}_{k}(K_{3}:2K_{2})\leq 3$.

To show the lower bound, let $\chi$ be the $k$-edge-coloring of
$K_{k+3}$. We first assume that there is no rainbow triangle in
$K_{k+3}$. From Lemma \ref{lemB2K2-2}, we have
\begin{eqnarray*}
\tau(k,K_{k+3},2K_{2})&\geq &\min\left\{2k,\tau(1,K_{4},2K_{2}),\frac{k^2-4}{2}+\tau(1,K_{3},2K_{2})\right\}\\[0.2mm]
&=&\min\{2k,3,(k^2-4)/2\}=3.
\end{eqnarray*}

Assume that there is a rainbow triangle in $v_1v_2v_3v_1$. Let
$F_1=K_{k+3}-v_1-v_2-v_3$, $\chi(v_1v_2)=1$, $\chi(v_2v_3)=2$, and
$\chi(v_3v_1)=3$. If there are at least two edges in $F_1$ with
color $1$ or $2$ or $3$, then the number of monochromatic $2K_2$ is
at least $2$, as desired.

Suppose that there is at most one edge, say $u_1u_2$ in $F_1$ with
color $1$ or $2$ or $3$. Let $m$ be the number of maximal
monochromatic stars in $F_1$.
\begin{claim}\label{Claim4}
There is a monochromatic copy of $2$-matching in $F_1$.
\end{claim}
\begin{proof}
Assume that there is no monochromatic $2$-matchings in $K_k$.
Clearly, the subgraph induced by $i \ (1\leq i\leq k-3)$ in
$F_1-u_1u_2$ must be either a star or a triangle. Suppose that there
are $m$ stars and $k-3-m$ triangles in $K_n$. Then $m\leq k-3$ and
the number of edges colored is at most
\begin{eqnarray*}
1+\sum_{i=1}^{m}(k-i)+3(k-3-m)&=&1+\frac{m(2k-m-1)+6(k-3-m)}{2}<\frac{k(k-1)}{2},
\end{eqnarray*}
a contradiction.
\end{proof}

From Claim \ref{Claim4}, there is a monochromatic copy of
$2$-matching in $F_1$, say with color $4$. Let $u_1w_1,u_2w_2$ be
the two edges of this matching and $\chi(u_1w_1)=\chi(u_2w_2)=4$.
Let $H$ be the subgraph induced by the vertices in
$\{u_1,w_1,u_2,w_2\}$. From Observation \ref{obs5-1}, if there are
three or four colors appearing on the edges of $H$, then there are
two rainbow triangles in $F_1$, as desired. Suppose that there are
exactly two colors appearing on the edges of $H$. Then there is a
chromatic $C_4$ or $P_4$. For the former, we are done. For the
latter, we can assume that $\chi(u_1u_2)=4$. If $\chi(u_1v_2)=4$,
then $\{u_1v_2,u_2w_2\},\{u_1w_1,u_2w_2\}$ form two monochromatic
$2$-matchings with color $4$. If $\chi(u_1v_2)=3$, then
$\{u_1v_2,v_1v_3\},\{u_1w_1,u_2w_2\}$ form two monochromatic
$2$-matchings. Therefore, we have $\chi(u_1v_2)=2$ or $1$.
Similarly, we have $\chi(u_1v_1)=3$ or $1$. Since $\chi(u_1u_2)=4$,
we can assume $\chi(u_1v_2)=\chi(u_1v_1)=1$. If $\chi(u_1v_3)\neq
1$, then we have a new rainbow triangle. So we assume that
$\chi(u_1v_3)=1$. Then $\{u_1v_3,v_1v_2\},\{u_1w_1,u_2w_2\}$ form
two monochromatic $2$-matchings, as desired.
\end{proof}

\begin{theorem}
For $k\geq 6$, we have
$$
\operatorname{GM}_{k}(K_{3}:nK_{2})\leq (2n-1)!!.
$$
\end{theorem}
\begin{proof}
Note that $\operatorname{gr}_{k}(K_{3}:nK_{2})=(k+1)(n-1)+2$. Let
$G$ be a $k$-edge-colored of $K_{(k+1)(n-1)+2}$ obtained from $k$
cliques $K_{2n}^1,K_{n-1}^2,...,K_{n-1}^{k}$ with colors
$1,2,....,k$, respectively. The \emph{reduced graph} $H$ of
$K_{(k+1)(n-1)+2}$ is constructed by contracting each clique
$K_{n-1}^1,K_{n-1}^2,...,K_{n-1}^{k}$ by a single vertex $u_{i} \
(1\leq i\leq k)$. Note that $H$ is a complete graph of order $k$ and
$V(H)=\{u_{i}\,|\,1\leq i\leq k\}$. Color the edges in
$\{u_iu_j\,|\,1\leq j\leq i-1\}$ incident $u_i \ (1\leq i\leq k)$
with color $i$ in $H$. This corresponds to color all the edges from
$K_{n-1}^{i}$ to $K_{n-1}^{j}$ with color $i$ in $K_{(k+1)(n-1)+2}$,
where $1\leq i\leq k$ and $1\leq j\leq i-1$, corresponding to the
edge $u_iu_j$ in $H$. For each $i \ (2\leq i\leq k)$, there is no
$nK_2$ colored by $i$ in $K_{(k+1)(n-1)+2}$. For the complete graph
$K_{2n}^1$, there are $(2n-1)!!$ monochromatic $n$-matchings with
color $1$.
\end{proof}

\section{Concluding Remark}

The \emph{Ramsey realization number} of $G$, written $rr(G)$, is the
number of different $2$-colorings of $K_{r(G)}$ which contain the
minimum number $M(G)$ of monochromatic $G$.

Harary and Prins \cite{HaPr74} proposed the following problem.
\begin{problem}
Which proper graphs $G$ have a unique Ramsey realization: $rr(G)=1$?
\end{problem}

The \emph{Gallai-Ramsey realization number} of $G$, written
$\operatorname{grr}_k(G)$, is the number of different exact
$k$-colorings of $K_{\operatorname{gr}_{k}(G,H)}$ which contain the
minimum total number $\operatorname{GM}_{k}(G,H)$ of rainbow $G$ and
monochromatic $H$.

Similarly to the problem introduced by Harary and Prins
\cite{HaPr74}, we introduce the following problem.
\begin{problem}
Which proper graphs $G$ have a unique Ramsey realization:
$\operatorname{grr}_k(G)=1$?
\end{problem}

\end{document}